\documentclass[reqno]{amsart}

\usepackage[margin=1in]{geometry}

\usepackage{mathtools}
\usepackage{amssymb, amsfonts}
\usepackage{mathrsfs}
\usepackage{microtype}
\usepackage{enumitem}

\numberwithin{equation}{section}

\theoremstyle{plain}
\newtheorem{theorem}{Theorem}[section]
\newtheorem{lemma}[theorem]{Lemma}
\newtheorem{proposition}[theorem]{Proposition}
\newtheorem{corollary}[theorem]{Corollary}

\theoremstyle{definition}

\theoremstyle{remark}

\usepackage[hidelinks]{hyperref}

\title[Prime plus a non-square-free integer]
{Every integer greater than 24 is the sum of a prime and a non-square-free integer}

\author{Peter J. Campbell}

\address{School of Mathematics and Physics, The University of Queensland, St Lucia, Brisbane, QLD 4072, Australia}
\email{p.campbell1@uq.edu.au}

\subjclass[2020]{Primary 11P32; Secondary 11N13}

\keywords{Goldbach-type problems, non-square-free integers, least primes in arithmetic progressions, exceptional zeros}

\begin{document}

\begin{abstract}
    Lee and O'Clarey proved that every odd integer \(n>24\) can be written as the sum of a prime and a positive integer that is not square-free. Under the generalised Riemann hypothesis, they established the same conclusion for every integer \(n>24\). We prove this result unconditionally, confirming their stated expectation. The proof combines an elementary congruence argument, finite computation, and explicit estimates for primes in arithmetic progressions.
\end{abstract}

\maketitle

\section{Introduction}\label{sec:intro}

Goldbach's conjecture asserts that every even integer greater than two can be written as the sum of two primes. A natural variation is to retain one prime summand while imposing a different arithmetic condition on the other. One such problem concerns \emph{square-free} integers, namely positive integers not divisible by the square of any prime.

Dudek~\cite{dudek2017_sum_prime_square_free} proved that every integer greater than two can be written as the sum of a prime and a square-free integer. Lee and O'Clarey~\cite{lee_oclarey2026_non_square_free} recently considered the complementary problem, requiring the second summand not to be square-free. They proved that every sufficiently large integer has such a representation~\cite[Theorem~1.1]{lee_oclarey2026_non_square_free}, although their argument is ineffective owing to its use of the Siegel--Walfisz theorem. They also proved that every odd integer \(n>24\) has such a representation, and that the same holds for all \(n>24\) under the generalised Riemann hypothesis~\cite[Theorems~1.2--1.3]{lee_oclarey2026_non_square_free}.

Lee and O'Clarey expected their conditional result to hold unconditionally. We prove that it does.

\begin{theorem}\label{thm:main}
    Every integer \(n>24\) can be written as
    \begin{equation*}
        n=p+r,
    \end{equation*}
    where \(p\) is prime and \(r\) is a positive integer that is not square-free.
\end{theorem}

Combining Theorem~\ref{thm:main} with direct verification for \(1\leq n\leq24\) gives the following corollary.

\begin{corollary}
    The only positive integers that cannot be written as the sum of a prime and a positive integer that is not square-free are
    \[
        1,2,3,4,5,8,24.
    \]
\end{corollary}

Our starting point is an elementary congruence argument used by Lee and O'Clarey in the proof of their conditional result. Let \(q\) be a prime not dividing \(n\), so that \(n\) lies in a reduced residue class modulo \(q^2\). If there exists a prime \(p < n\) such that \(p\equiv n\pmod{q^2}\), then \(q^2\mid n-p\), so \(n-p\) is a positive integer that is not square-free. In the intermediate range of their proof, Lee and O'Clarey use the fact that \(n\) is smaller than the product of the first twenty primes to ensure that some prime \(q\leq71\) does not divide \(n\). They then obtain the required prime \(p\) using a conditional bound of Lamzouri, Li and Soundararajan~\cite{lamzouri_li_soundararajan2015_quadratic_nonresidue} for the least prime in an arithmetic progression.

Building on this observation, we take \(q\) to be the least prime not dividing \(n\). This choice forces every prime smaller than \(q\) to divide \(n\), so their product gives a lower bound for \(n\). Unlike the counting-function approach used by Lee and O'Clarey in their unconditional argument, our method reduces the problem to finding a prime in a single residue class modulo \(q^2\) for each \(n\) outside a finite range of small integers treated separately. For every prime \(q\geq17\), we show that each reduced residue class modulo \(q^2\) contains a prime smaller than this product. Applying this to the residue class of \(n\) then gives the required representation.

For \(17\leq q\leq14\,347\), we establish this comparison by computing the least prime in each reduced residue class modulo \(q^2\). For \(q\geq14\,369\), we use explicit estimates of Bennett, Martin, O'Bryant, and Rechnitzer~\cite{bennett_martin_obryant_rechnitzer2018_explicit_bounds} for primes in arithmetic progressions. Since \(14\,347\) and \(14\,369\) are consecutive primes, these ranges cover every prime \(q\geq17\). The smaller primes \(q\leq13\), for which this comparison fails, are handled separately by direct congruence arguments and finite verification.

The remainder of the paper is organised as follows. In Section~\ref{sec:reduction}, we develop the least-missing-prime criterion and give the elementary and computational arguments. In Section~\ref{sec:analytic}, we establish an explicit analytic bound for all larger primes and complete the proof of Theorem~\ref{thm:main}.

\section{A Least-Missing-Prime Criterion}\label{sec:reduction}

For a given positive integer \(n\), let
\begin{equation*}
    q=q(n)
    :=
    \min\{p \text{ prime} : p\nmid n\}.
\end{equation*}
This choice gives two properties central to our argument: every prime smaller than \(q\) divides \(n\), while \((n,q^2)=1\). The first provides a lower bound for \(n\), while the second allows us to work in a reduced residue class modulo \(q^2\).

We begin with the cases \(q=2\) and \(q=3\), which correspond respectively to odd integers and even integers not divisible by \(3\). These cases provide simple, concrete illustrations of the congruence argument before we consider larger primes.

\begin{lemma}\label{lem:initial_q}
    Let \(n\) be a positive integer. If \(q=2\) and \(n>5\), or if \(q=3\) and \(n>8\), then \(n\) can be written as
    \begin{equation*}
        n=p+r,
    \end{equation*}
    where \(p\) is prime and \(r\) is a positive integer that is not square-free.
\end{lemma}

\begin{proof}
    In either case, it suffices to find a prime \(p < n\) satisfying \(p\equiv n\pmod{q^2}\), since then \(n-p\) is a positive multiple of \(q^2\) and is therefore not square-free.

    Suppose first that \(q=2\) and \(n > 5\). Then \(2\nmid n\), so \(n\) lies in one of the two reduced residue classes modulo \(4\). For each possible residue class, the following table gives the least prime \(p\) satisfying \(p\equiv n\pmod{4}\).
    \begin{equation*}
        \begin{array}{c|c}
            n\bmod 4 & p\\
            \hline
            1 & 5\\
            3 & 3
        \end{array}
    \end{equation*}
    In either case, \(p\leq 5 < n\) and \(4 \mid n-p\). Therefore, \(r=n-p\) is positive and not square-free, giving \(n=p+r\).
    
    Now suppose \(q=3\) and \(n > 19\). By the definition of \(q\), we have \(2\mid n\) and \(3\nmid n\). Hence \(n\) lies in one of the six reduced residue classes modulo \(9\). As above, the following table gives the least prime \(p\) congruent to \(n\) in each possible residue class.
    \begin{equation*}
        \begin{array}{c|c}
            n\bmod 9 & p\\
            \hline
            1 & 19\\
            2 & 2\\
            4 & 13\\
            5 & 5\\
            7 & 7\\
            8 & 17
        \end{array}
    \end{equation*}
    Each listed prime satisfies \(p\leq19<n\), so the table proves the result when \(n>19\). The remaining integers \(n>8\) satisfying \(q=3\) are \(10\), \(14\), and \(16\), for which we have
    \begin{equation*}
        10=2+8,\qquad
        14=5+9,\qquad
        16=7+9.
    \end{equation*}
    This completes the proof.
\end{proof}

The lower bounds in Lemma~\ref{lem:initial_q} are sharp, since neither \(5\) nor \(8\) can be written as the sum of a prime and a positive integer that is not square-free.

The case \(q=2\) recovers the result of Lee and O'Clarey~\cite[Theorem~1.2]{lee_oclarey2026_non_square_free} for odd integers \(n>24\). More generally, Lemma~\ref{lem:initial_q} shows that every integer \(n>24\) satisfying \(6\nmid n\) has the required representation.

A further elementary refinement is worth recording. If \(n>2\) and \(n\equiv2\pmod4\), then taking \(p=2\) gives
\begin{equation*}
    4\mid n-p,
\end{equation*}
so \(n-p\) is positive and not square-free. More generally, if
\begin{equation*}
    n\equiv p \pmod{p^2}
\end{equation*}
for some prime \(p < n\), then \(p^2\mid n-p\). For instance, after Lemma~\ref{lem:initial_q} it remains only to consider multiples of \(6\). Among these, taking \(p=2\) eliminates \(6,18,30\pmod{36}\), while taking \(p=3\) eliminates \(12\) and \(30\pmod{36}\). Thus, among multiples of \(6\), these congruence arguments leave only the residue classes \(0,24\pmod{36}\) to be considered.

We do not pursue such refinements systematically. Instead, we develop a uniform criterion based on the least prime not dividing \(n\), using finite computation and explicit estimates for primes in arithmetic progressions to establish the required bounds. We retain this observation to illustrate that direct congruence arguments can provide additional information beyond the divisibility conditions obtained below.

To express the corresponding divisibility condition for larger values of \(q\), we write
\begin{equation*}
    x^-\#
    :=
    \prod_{p<x}p
\end{equation*}
for the product of the primes strictly smaller than \(x\). Since \(q\) is the least prime not dividing \(n\), every prime smaller than \(q\) divides \(n\). Equivalently,
\begin{equation*}
    q^-\#\mid n.
\end{equation*}
In particular, when \(q=5\), this recovers the divisibility condition \(6\mid n\) obtained above.

The tables in the proof of Lemma~\ref{lem:initial_q} motivate the following definition. For integers \(m\geq 2\) and \(a\) with \((a,m)=1\),
let
\begin{equation*}
    P(a,m)
    :=
    \min\{p\text{ prime}:p\equiv a\pmod m\}.
\end{equation*}
By Dirichlet's theorem on primes in arithmetic progressions, \(P(a,m)\) is well defined. Define
\begin{equation*}
    P(m)
    :=
    \max_{\substack{a\bmod m\\(a,m)=1}}P(a,m).
\end{equation*}
Thus \(P(q^2)\) is the largest of the least-prime representatives among the reduced residue classes modulo \(q^2\). In particular,
\begin{equation*}
    P(2^2)=5
    \qquad\text{and}\qquad
    P(3^2)=19.
\end{equation*}

For \(q=5\), the corresponding least-prime representatives are as follows:
\begin{equation*}
    \begin{array}{c|r}
        n\bmod 25 & p\\
        \hline
        1  & 101\\
        2  & 2\\
        3  & 3\\
        4  & 29\\
        6  & 31\\
        7  & 7\\
        8  & 83\\
        9  & 59\\
        11 & 11\\
        12 & 37
    \end{array}
    \qquad\qquad
    \begin{array}{c|r}
        n\bmod 25 & p\\
        \hline
        13 & 13\\
        14 & 89\\
        16 & 41\\
        17 & 17\\
        18 & 43\\
        19 & 19\\
        21 & 71\\
        22 & 47\\
        23 & 23\\
        24 & 149
    \end{array}
\end{equation*}
and hence
\begin{equation*}
    P(5^2)=149.
\end{equation*}

Thus every integer \(n>149\) satisfying \(q=5\) has the required representation. More generally, the individual least-prime representatives in the tables will no longer be important. Only their maximum \(P(q^2)\) matters. This gives the basic criterion underlying the rest of the argument.

\begin{proposition}[Least-missing-prime criterion]\label{prop:least_prime_criterion}
    Let \(n\) be a positive integer, and let \(q=q(n)\) be the least prime not dividing \(n\). If \(P(q^2)<n\), then \(n\) can be written as \(n=p+r\), where \(p\) is prime and \(r\) is a positive integer that is not square-free.
\end{proposition}

\begin{proof}
    Since \(q\nmid n\), the residue class of \(n\) modulo \(q^2\) is reduced. Let \(p\) be the least prime in this class. Then \(p\leq P(q^2)<n\), and \(q^2\mid n-p\). Hence \(r=n-p\) is positive and not square-free.
\end{proof}

Since \(q^-\#\mid n\), Proposition~\ref{prop:least_prime_criterion} applies whenever \(P(q^2)<q^-\#\).

\begin{lemma}\label{lem:P_values}
    For \(q=5,7,11,\) and \(13\), the values of \(P(q^2)\) and \(q^-\#\) are
    \begin{equation*}
        \begin{array}{c|r|r}
            q & P(q^2) & q^-\#\\
            \hline
            5  & 149  & 6\\
            7  & 613  & 30\\
            11 & 1847 & 210\\
            13 & 3617 & 2310
        \end{array}
    \end{equation*}
\end{lemma}

\begin{proof}
    These values are obtained by computing the least prime in each reduced residue class modulo \(q^2\), as implemented in \texttt{compute\_small\_prime\_square\_values.py} in~\cite{PrimePlusNonsquarefreeRepo}.
\end{proof}

For these four primes, the table shows that \(P(q^2)\geq q^-\#\), so the comparison \(P(q^2)<q^-\#\) fails. Nevertheless, \(P(q^2)\leq3617\) in each case. Hence Proposition~\ref{prop:least_prime_criterion} gives the required representation whenever \(5\leq q\leq13\) and \(n>3617\). The remaining integers are handled by finite verification.

\begin{lemma}\label{lem:finite_verification}
    Every integer \(n\) satisfying
    \begin{equation*}
        24<n\leq3617
    \end{equation*}
    can be written as
    \begin{equation*}
        n=p+r,
    \end{equation*}
    where \(p\) is prime and \(r\) is a positive integer that is not square-free.
\end{lemma}

\begin{proof}
    By Lemma~\ref{lem:initial_q}, it remains only to consider integers \(n\) satisfying
    \begin{equation*}
        24<n\leq3617
        \qquad\text{and}\qquad
        6\mid n.
    \end{equation*}
    There are \(598\) such integers. For each of them, we searched over the primes \(p<n\) and found a prime for which \(n-p\) is divisible by the square of a prime. The resulting representations are produced by \texttt{verify\_small\_representations.py} in~\cite{PrimePlusNonsquarefreeRepo}.
\end{proof}

The conclusion of Lemma~\ref{lem:finite_verification} also follows from Lee and O'Clarey~\cite[Lemma~2.1]{lee_oclarey2026_non_square_free}, who verified the result computationally for every integer \(24<n\leq8\cdot10^9\).

For every prime \(17\leq q\leq14\,347\), the required comparison is established by finite computation.

\begin{lemma}\label{lem:computational_range}
    For every prime \(q\) satisfying
    \(17\leq q\leq14\,347\), we have
    \begin{equation*}
        P(q^2)<q^-\#.
    \end{equation*}
\end{lemma}

\begin{proof}
    For each prime \(q\) in the stated range, the script \texttt{verify\_prime\_square\_bound.py} in~\cite{PrimePlusNonsquarefreeRepo} computes the least prime in every reduced residue class modulo \(q^2\). The computation uses a segmented sieve to find prime representatives, after which the least prime in each remaining residue class is found by testing successive candidates using trial division. In every case, the resulting value of \(P(q^2)\) satisfies \(P(q^2)<q^-\#\). The computation took just under six hours on a desktop PC with an AMD Ryzen 5 5600X processor (3.70 GHz base clock). The complete computational results are available in the repository.
\end{proof}

\section{An Explicit Analytic Bound}\label{sec:analytic}

The preceding section handles every integer \(n>24\) whose least missing prime satisfies \(q\leq14\,347\). It remains to establish
\begin{equation}\label{eq:analytic_criterion}
    P(q^2)<q^-\#
\end{equation}
for every prime \(q\geq14\,369\), using explicit estimates of Bennett, Martin, O'Bryant, and Rechnitzer~\cite{bennett_martin_obryant_rechnitzer2018_explicit_bounds} for primes in arithmetic progressions.

For an integer \(m \geq 3\), an integer \(a\) with \((a,m)=1\), and a real number \(x \geq 2\), define
\begin{equation*}
    \theta(x;m,a)
    :=
    \sum_{\substack{p\le x\\p\equiv a\pmod m}}\log p
\end{equation*}
and
\begin{equation*}
    \psi(x;m,a)
    :=
    \sum_{\substack{p^k\le x\\p^k\equiv a\pmod m\\k\ge1}}
    \log p.
\end{equation*}
Thus \(\psi(x;m,a)-\theta(x;m,a)\) is the contribution from proper prime powers. Since each prime contributes a positive weight,
\begin{equation*}
    \theta(x;m,a)>0
    \quad\Longleftrightarrow\quad
    P(a,m)\le x.
\end{equation*}
We shall obtain a lower bound for \(\theta(x;m,a)\) by estimating \(\psi(x;m,a)\) and bounding the contribution from proper prime powers.

We write \(L(s,\chi)\) for the Dirichlet \(L\)-function associated with a character \(\chi\). Following \cite[Definition~6.1]{bennett_martin_obryant_rechnitzer2018_explicit_bounds}, set
\begin{equation*}
    R:=9.645908801.
\end{equation*}
A real zero \(\beta_0\) of a Dirichlet \(L\)-function modulo \(m\) is called \emph{exceptional}, relative to \(m\), if
\begin{equation*}
    1-\frac{1}{R\log m}\le\beta_0<1.
\end{equation*}
As noted by Bennett et al.~\cite[Definitions~3.1 and~6.1]{bennett_martin_obryant_rechnitzer2018_explicit_bounds}, for \(m\geq10^5\), there is at most one such zero among all Dirichlet \(L\)-functions modulo \(m\), and it can occur only for a real nonprincipal character. We do not assume that this zero is absent; instead, we bound its possible contribution explicitly.

We closely follow the argument of Bennett et al. in their derivation of Lemma~6.12~\cite{bennett_martin_obryant_rechnitzer2018_explicit_bounds}, which combines estimates for \(\psi(x;m,a)\) and \(\psi(x;m,a)-\theta(x;m,a)\) with an explicit bound on the exceptional zero. However, we retain the dependence on \(\beta_0\), allowing us to exploit the particular structure of prime-square moduli.

\begin{lemma}\label{lem:theta_lower_bound}
    Let \(m\geq10^5\) be an integer and let \(a\) be an integer satisfying \((a,m)=1\). Suppose that \(\log x\geq4R(\log m)^2\). Then
    \begin{equation}\label{eq:theta_lower_bound}
        \theta(x;m,a)
        >
        \frac{x}{\varphi(m)} -\frac{1.012x^{\beta_0}}{\varphi(m)} -1.4579x\sqrt{\frac{\log x}{R}} \exp\left(-\sqrt{\frac{\log x}{R}}\right) -1.001\sqrt{x},
    \end{equation}
    where the term \(-1.012x^{\beta_0}/\varphi(m)\) is omitted if no exceptional zero exists modulo \(m\).
\end{lemma}

\begin{proof}
    By \cite[Lemma~6.10]{bennett_martin_obryant_rechnitzer2018_explicit_bounds},
    \begin{equation*}
        \psi(x;m,a)
        \geq
        \frac{x}{\varphi(m)} -\frac{1.012x^{\beta_0}}{\varphi(m)} -1.4579x\sqrt{\frac{\log x}{R}} \exp\left(-\sqrt{\frac{\log x}{R}}\right),
    \end{equation*}
    where the term \(-1.012x^{\beta_0}/\varphi(m)\) is omitted if no exceptional zero exists.

    Since \(\log x\geq4R(\log m)^2\) and \(m\geq10^5\), we have \(x\geq10^{500}\). Hence \cite[Lemma~6.11]{bennett_martin_obryant_rechnitzer2018_explicit_bounds} gives
    \begin{equation*}
        0\leq\psi(x;m,a)-\theta(x;m,a)
        <
        1.001\sqrt{x}.
    \end{equation*}
    Combining these estimates yields the result.
\end{proof}

To bound the possible contribution from an exceptional zero in Lemma~\ref{lem:theta_lower_bound}, we use the following result of Bennett et al. For moduli of the form \(q^2\), where \(q\) is an odd prime, we can strengthen their bound by showing that the associated real nonprincipal character is induced by a character modulo \(q\).

\begin{proposition}[{\cite[Proposition~1.11]{bennett_martin_obryant_rechnitzer2018_explicit_bounds}}]\label{prop:exceptional_zero_bound}
    Let \(d \geq 3\) be an integer and let \(\chi\) be a quadratic character modulo \(d\). If \(\beta > 0\) is a real zero of \(L(s,\chi)\), then
    \begin{equation}\label{eq:exceptional_zero_bound}
        \beta
        \leq
        1-\frac{40}{\sqrt{d}(\log d)^2}.
    \end{equation}
\end{proposition}

\begin{lemma}\label{lem:exceptional_zero_prime_square}
    Let \(q\) be an odd prime and let \(\chi\) be a real nonprincipal character modulo \(q^2\). If \(\beta>0\) is a real zero of \(L(s,\chi)\), then
    \begin{equation}\label{eq:prime_square_zero_bound}
        \beta
        \leq
        1-\frac{40}{\sqrt{q}(\log q)^2}.
    \end{equation}
\end{lemma}

\begin{proof}
    Let \(a\) be an integer. By the binomial theorem,
    \begin{equation*}
        (a+q)^q
        =
        a^q+q^2a^{q-1} + \sum_{j=2}^{q}\binom{q}{j}a^{q-j}q^j.
    \end{equation*}
    Hence
    \begin{equation*}
        (a+q)^q\equiv a^q\pmod{q^2}.
    \end{equation*}
    Since \(\chi\) is periodic modulo \(q^2\) and completely multiplicative, it follows that
    \begin{equation*}
        \chi(a+q)^q=\chi(a)^q.
    \end{equation*}
    As \(\chi\) is real, its values lie in \(\{-1,0,1\}\). Since \(q\) is odd, we have
    \begin{equation*}
        \chi(a+q)=\chi(a).
    \end{equation*}
    Thus \(\chi\) is periodic modulo \(q\).

    Since \(\chi(k)=0\) precisely when \(q\mid k\), the periodicity established above, together with complete multiplicativity, shows that \(\chi\) defines a real Dirichlet character \(\chi^*\) modulo \(q\). This character is nonprincipal, since otherwise \(\chi\) would be principal. Being real and nonprincipal, \(\chi^*\) is therefore quadratic. By construction, \(\chi(k)=\chi^*(k)\) for every integer \(k\), and hence
    \begin{equation*}
        L(s,\chi)=L(s,\chi^*)
    \end{equation*}
    for \(\Re(s)>1\), and therefore everywhere by analytic continuation. In particular, \(L(\beta,\chi)=0\) implies \(L(\beta,\chi^*)=0\). Applying Proposition~\ref{prop:exceptional_zero_bound} with \(d=q\) gives
    \begin{equation*}
        \beta
        \leq
        1-\frac{40}{\sqrt{q}(\log q)^2},
    \end{equation*}
    as required.
\end{proof}

Combining Lemmas~\ref{lem:theta_lower_bound} and~\ref{lem:exceptional_zero_prime_square} gives the following sufficient condition for an upper bound on \(P(q^2)\).

\begin{lemma}\label{lem:prime_square_least_prime_bound}
    Let \(q\) be an odd prime with \(q^2\geq10^5\), and let \(x\) be a real number satisfying
    \begin{equation}\label{eq:prime_square_size_condition}
        \log x\geq16R(\log q)^2.
    \end{equation}
    If
    \begin{equation}\label{eq:prime_square_positivity}
            1.012\exp\left(-\frac{40\log x}{\sqrt{q}(\log q)^2} \right) + 1.4579q(q-1)\sqrt{\frac{\log x}{R}} \exp\left(-\sqrt{\frac{\log x}{R}}\right) + \frac{1.001q(q-1)}{\sqrt{x}}
            <
            1,
    \end{equation}
    then \(P(q^2)\leq x\).
\end{lemma}

\begin{proof}
    Let \(a\) be an integer satisfying \((a,q)=1\). Since
    \begin{equation*}
        4R(\log(q^2))^2=16R(\log q)^2,
    \end{equation*}
    Lemma~\ref{lem:theta_lower_bound} applies with \(m=q^2\).

    If an exceptional zero \(\beta_0\) exists modulo \(q^2\), then, since \(q^2\geq10^5\), it is a positive real zero associated with a real nonprincipal character. Hence Lemma~\ref{lem:exceptional_zero_prime_square}
    gives
    \begin{equation*}
        x^{\beta_0}
        \leq
        x\exp\left(-\frac{40\log x}{\sqrt{q}(\log q)^2}\right).
    \end{equation*}
    Consequently, using \(\varphi(q^2)=q(q-1)\), we obtain
    \begin{equation*}
        \begin{aligned}
            \frac{q(q-1)}{x}\theta(x;q^2,a)
            &>
            1-1.012\exp\left(
                -\frac{40\log x}{\sqrt{q}(\log q)^2}
            \right)\\
            &\qquad
            -1.4579q(q-1)\sqrt{\frac{\log x}{R}}
            \exp\left(-\sqrt{\frac{\log x}{R}}\right)\\
            &\qquad
            -\frac{1.001q(q-1)}{\sqrt{x}}
            >0,
        \end{aligned}
    \end{equation*}
    where the first inequality remains valid if no exceptional zero exists.

    Thus \(P(a,q^2)\leq x\). Since \(a\) was arbitrary, it follows that \(P(q^2)\leq x\).
\end{proof}

We shall apply Lemma~\ref{lem:prime_square_least_prime_bound} with \(x=q^-\#\) to establish \eqref{eq:analytic_criterion} for every prime \(q\geq14\,369\).

\begin{proposition}\label{prop:explicit_analytic_range}
    For every prime \(q\geq14\,369\), we have
    \begin{equation*}
        P(q^2)<q^-\#.
    \end{equation*}
\end{proposition}

\begin{proof}
    Set \(x=q^-\#\), so that
    \begin{equation*}
        \log x=\sum_{p<q}\log p,
    \end{equation*}
    and write \(u:=\sqrt{\log x/R}\). Since \(q\) is odd and \(q^2\geq10^5\), it suffices by Lemma~\ref{lem:prime_square_least_prime_bound} to verify \eqref{eq:prime_square_size_condition} and \eqref{eq:prime_square_positivity}.

    Write
    \begin{equation*}
        \vartheta(t):=\sum_{p\leq t}\log p.
    \end{equation*}
    By \cite[Corollary to Theorem~6]{rosser_schoenfeld1975_chebyshev_bounds}, we have
    \begin{equation*}
        \vartheta(t) > 0.985t
        \qquad(t\geq11\,927).
    \end{equation*}
    Since \(q\) is prime,
    \begin{equation*}
        \begin{aligned}
            \log x-16R(\log q)^2
            &=\vartheta(q)-\log q-16R(\log q)^2\\
            &>0.985q-\log q-16R(\log q)^2\\
            &>0,
        \end{aligned}
    \end{equation*}
    where the final inequality holds for \(q\geq14\,369\), as the expression is increasing in this range and exceeds \(0.82\) at \(q=14\,369\). This establishes \eqref{eq:prime_square_size_condition}.
    
    It remains to verify \eqref{eq:prime_square_positivity}. For \(q\geq14\,369\), we have
    \begin{equation*}
        \log x>0.985q-\log q>0.984q.
    \end{equation*}

    Consequently, the first term in \eqref{eq:prime_square_positivity} satisfies
    \begin{equation*}
        1.012\exp\left(
            -\frac{40\log x}{\sqrt q(\log q)^2}
        \right)
        <
        1.012\exp\left(
            -\frac{39.36\sqrt q}{(\log q)^2}
        \right).
    \end{equation*}

    Furthermore, the condition \(\log x\geq16R(\log q)^2\) gives \(u\geq4\log q\). Since \(ue^{-u}\) is decreasing for \(u>1\),
    \begin{equation*}
        \begin{aligned}
            1.4579q(q-1)ue^{-u}
            &\leq
            1.4579q(q-1)\frac{4\log q}{q^4}\\
            &\leq
            \frac{5.8316\log q}{q^2}.
        \end{aligned}
    \end{equation*}
    Finally,
    \begin{equation*}
        \frac{1.001q(q-1)}{\sqrt{x}}
        <
        1.001q^2e^{-0.492q}.
    \end{equation*}

    Therefore, the left-hand side of \eqref{eq:prime_square_positivity} is bounded above by
    \begin{equation*}
        1.012\exp\left(-\frac{39.36\sqrt q}{(\log q)^2}\right) +\frac{5.8316\log q}{q^2} + 1.001q^2e^{-0.492q}.
    \end{equation*}
    Each of these three terms is decreasing for \(q\geq14\,369\). Evaluating at \(q=14\,369\), their sum is less than
    \begin{equation*}
        3\cdot10^{-7}<1.
    \end{equation*}
    Thus \eqref{eq:prime_square_positivity} also holds throughout this range.

    We have therefore verified both hypotheses of Lemma~\ref{lem:prime_square_least_prime_bound} for every prime \(q\geq14\,369\), giving \(P(q^2)\leq q^-\#\). Since \(P(q^2)\) is prime and \(q^-\#\) is composite, equality is impossible, completing the proof.
\end{proof}

We are now in a position to prove our main theorem.

\begin{proof}[Proof of Theorem~\ref{thm:main}]
    Let \(n>24\). If \(n\leq3617\), then the result follows from Lemma~\ref{lem:finite_verification}. We may therefore suppose that \(n>3617\).

    Let \(q=q(n)\) be the least prime not dividing \(n\). If \(q=2\) or \(q=3\), then the result follows from Lemma~\ref{lem:initial_q}.

    If \(5\leq q\leq13\), then Lemma~\ref{lem:P_values} gives
    \begin{equation*}
        P(q^2)\leq3617<n,
    \end{equation*}
    so Proposition~\ref{prop:least_prime_criterion} gives the required representation.

    Now suppose that \(q\geq17\). Lemma~\ref{lem:computational_range} establishes \(P(q^2)<q^-\#\) for \(17\leq q\leq14\,347\), while Proposition~\ref{prop:explicit_analytic_range} establishes the same inequality for \(q\geq14\,369\). Since \(14\,347\) and \(14\,369\) are consecutive primes, these ranges cover every prime \(q\geq17\).

    Every prime smaller than \(q\) divides \(n\), so \(q^-\#\mid n\). Consequently,
    \begin{equation*}
        P(q^2)<q^-\#\leq n.
    \end{equation*}
    Proposition~\ref{prop:least_prime_criterion} again gives the required representation.
\end{proof}

\section{Acknowledgments}
The author thanks Adrian Dudek for suggesting the problem, for helpful discussions concerning future directions, and for comments on an earlier version of the manuscript. The author also thanks Timothy Trudgian for drawing his attention to the explicit estimates of Bennett et al.

\bibliographystyle{abbrv}
\bibliography{refs}

@article{dudek2017_sum_prime_square_free,
    author  = {Dudek, Adrian W.},
    title   = {{On the Sum of a Prime and a Square-Free Number}},
    journal = {Ramanujan J.},
    year    = {2017},
    volume  = {42},
    number  = {1},
    pages   = {233--240},
    doi     = {10.1007/s11139-015-9736-2},
}

@article{lee_oclarey2026_non_square_free,
    author  = {Lee, E. S. and O'Clarey, R.},
    title   = {{On the Sum of a Prime and a Number That Is Not Square-Free}},
    journal = {Arch. Math.},
    year    = {2026},
    volume  = {127},
    pages   = {153--163},
    doi     = {10.1007/s00013-026-02275-6},
}

@misc{PrimePlusNonsquarefreeRepo,
  author       = {Campbell, Peter J.},
  title        = {{Computational verification for the sum of a prime and a non-square-free number}},
  year         = {2026},
  howpublished = {GitHub, version 2.0.0,
                  \url{https://github.com/PeterJCampbell1/prime-plus-nonsquarefree/releases/tag/v2.0.0}},
}

@article{lamzouri_li_soundararajan2015_quadratic_nonresidue,
    author  = {Lamzouri, Youness and Li, Xiannan and Soundararajan, Kannan},
    title   = {{Conditional Bounds for the Least Quadratic Non-Residue and Related Problems}},
    journal = {Math. Comp.},
    year    = {2015},
    volume  = {84},
    number  = {295},
    pages   = {2391--2412},
}

@article{bennett_martin_obryant_rechnitzer2018_explicit_bounds,
    author  = {Bennett, Michael A. and Martin, Greg and O'Bryant, Kevin and Rechnitzer, Andrew},
    title   = {{Explicit Bounds for Primes in Arithmetic Progressions}},
    journal = {Illinois J. Math.},
    year    = {2018},
    volume  = {62},
    number  = {1--4},
    pages   = {427--532},
    doi     = {10.1215/ijm/1552442669},
}

@article{rosser_schoenfeld1975_chebyshev_bounds,
    author  = {Rosser, J. Barkley and Schoenfeld, Lowell},
    title   = {{Sharper Bounds for the Chebyshev Functions $\theta(x)$ and $\psi(x)$}},
    journal = {Math. Comp.},
    year    = {1975},
    volume  = {29},
    number  = {129},
    pages   = {243--269},
    doi     = {10.1090/S0025-5718-1975-0457373-7},
}

\end{document}